\documentclass{article}
\usepackage{amsfonts}
\usepackage{amssymb}
\usepackage{amsthm}
\usepackage{amsmath}
\usepackage{amsopn}

\theoremstyle{plain}
\newtheorem{thm}{Theorem}[section]
\newtheorem*{thma}{Theorem A}
\newtheorem*{thma'}{Theorem A$'$}
\newtheorem*{thmb}{Theorem B}
\newtheorem*{thmb'}{Theorem B$'$}
\newtheorem{prop}[thm]{Proposition}
\newtheorem{lem}[thm]{Lemma}
\newtheorem{cor}[thm]{Corollary}

\newtheorem{defn}{Definition}

\newtheorem{obs}{Observation}

\newcommand{\divhull}[1]{\langle {#1}\rangle_{\mathbb{Q}}}

\newcommand{\ek}[1]{k(( {#1})) }

\begin{document}

\title{Real closed exponential fields}

\author{P.\ D'Aquino, J.\ F.\ Knight, S.\ Kuhlmann, and K.\ Lange}

\maketitle


\abstract{In an extended abstract \cite{R}, Ressayre considered real closed exponential fields and integer parts that respect the exponential function.  He outlined a proof that every real closed exponential field has an exponential integer part.  In the present paper, we give a detailed account of Ressayre's construction.  The construction becomes canonical once we fix the real closed exponential field $R$, a residue field section, and a well ordering
$<$.  The construction is clearly constructible over these objects.  Each step looks effective, but it may require many steps.  We produce an example of an exponential field $R$ with a residue field $k$ and a well ordering $<$ such that $D^c(R)$ is low and $k$ and $<$ are $\Delta^0_3$, and Ressayre's construction cannot be completed in $L_{\omega_1^{CK}}$.}

\section{Introduction}

\begin{defn}

A \emph{real closed field} is an ordered field in which every non-negative element is a square, and every odd degree polynomial has a root.

\end{defn}

Tarski's celebrated elimination of quantifiers \cite{Tarski} shows that the axioms for real closed fields generate the complete theory
of the ordered field of reals, so this theory is decidable.

\begin{defn} [Integer part]
An integer part of an ordered  field $R$ is a discretely ordered subring $Z$ such that for each $r\in R$, there exists $z\in Z$ with $z\leq r < z+1$.
\end{defn}

If $R$ is Archimedean, then $\mathbb{Z}$ is the unique integer part.  In general, the integer part for $R$ is not unique.  Shepherdson \cite{S} showed that  a discrete ordered ring $Z$ is an integer parts of some real closed fields if and only if $Z$ satisfies a fragment of Peano Arithmetic called {\em Open Induction}.  Open Induction is the first order theory, in the language $\mathcal{L}=\{+, \cdot, <, 0, 1\}$, of discretely ordered commutative rings with a multiplicative identity $1$ whose set of non-negative elements satisfies, for each quantifier-free formula $\Phi(x, y)$, the associated induction axiom $I(\Phi)$:
\begin{equation*}
(\forall y)[\Phi(0, y) \ \&\ (\forall x)[\Phi(x, y)\rightarrow \Phi(x+1, y)]\rightarrow (\forall x)\ \Phi(x, y)].
\end{equation*}




We now consider real closed fields.   In \cite{MR}, Mourgues and Ressayre proved the following.

\begin{thma} Every real closed field has an integer part.
\end{thma}

 In an extended abstract \cite{R}, Ressayre outlines the proof for  an  analogue of Theorem A  for exponential real closed fields:

\begin{defn}\label{D:EXP}
A real closed exponential field is a real closed field $R$ endowed with an isomorphism of ordered groups:
\begin{center}
$(R,+,0,<) \rightarrow (R^{>0},\cdot ,1,<)$
\vspace{.1in}

$x\mapsto 2^x$
\end{center}
where $(R,+,0,<) $ is the additive  group of $R$ and $(R^{>0},\cdot ,1,<)$ is the multiplicative group of positive elements  of $R$. That is,  $2^x$ satisfies the following axioms:

\vspace{.2cm}
\noindent \it{1.} $2^{x+y} = 2^x2^y$,

\vspace{.2cm}
\noindent  \it{2.}  $x < y$ implies $2^x < 2^y$,

\vspace{.2cm}
\noindent  \it{3.} for all $x > 0$, there exists $y$ such that $2^y = x$; i.e., $\log(x)$ is defined.

\vspace{.2cm}
\noindent  We require also the following:

\vspace{.2cm}
\noindent \it{4.} $2^1 = 2$,

\vspace{.2cm}
\noindent  \it{5.}   for all $x\in R$, $x>n^2$ implies  $2^x>x^n$  \hspace{.1in}($n\geq 1$).
\end{defn}

\vspace{.2cm}

We now consider integer parts closed under exponentiation in real closed exponential fields.

\begin{defn} [Exponential integer part]
Let $R$ be a real closed exponential field.  An \emph{exponential integer part} is an integer part $Z$ such that for all positive $z\in Z$, we have $2^z\in Z$.
\end{defn}

\noindent
We observe that $\mathbb{Z}$ is an exponential integer part for any Archimedean real closed field.  In an extended abstract \cite{R}, Ressayre outlined the proof for the analogue of Theorem A for real closed exponential fields.

\begin{thmb} [Ressayre]
\label{Ressayre}

If $R$ is a real closed exponential field, then $R$ has an exponential integer part.

\end{thmb}

In this paper we revisit Ressayre's  extended abstract, providing the details of the proofs, and focusing on the complexity of the construction.
In \S \ref{S:AlgPrelim}, we give the necessary  algebraic preliminaries. In \S \ref{S:MRconst}, we briefly outline Mourgues and Ressayre's construction of an integer part for a real closed field.  In \S \ref{S:Rconst}, we provide the details of Ressayre's construction of an exponential integer part for a real closed exponential field.  The construction is canonical with respect to a given real closed field, a residue field section, and a well ordering of the elements of the real closed field.  In \S \ref{S:compEIP}, we look at the complexity of Ressayre's construction.  We produce a low real closed exponential field $R$, with a $\Delta^0_3$ residue field section $k$ and a $\Delta^0_3$ ordering $<$ of type $\omega+\omega$, so that Ressayre's construction applied to these inputs is not completed in the least admissible set.  

Since our $R$ is recursively saturated, there is another exponential integer part that is $\Sigma^0_2$.  In general, for a countable real closed exponential field, we may use $\Sigma$-saturation (an old notion due to Ressayre), to produce an exponential integer part $Z$ such that $\omega_1^{(R,Z)} = \omega_1^Z$; i.e., $(R,Z)$ is an element of  a fattening of the least admissible set over $R$.


\section{Algebraic preliminaries}\label{S:AlgPrelim}

In this section, we give some algebraic background for the construction Mourgues and Ressayre. We recall the natural valuation on an ordered field $R$.

\begin{defn} [Archimedean equivalence]

For $x,y\in R^{\times}:=R\setminus \{0\}$, $x\sim y$ iff there exists $n\in \mathbb N$ such that $n|x| \geq |y|$ and $n |y|\geq |x|$, where $|x|:= \mbox{ maxÊ} \{ x, -x \}.$    We denote the equivalence class of $x\in R$ by $v(x)$.

\end{defn}


\begin{defn} [Value group]
 The {\em value group of $R$} is  the set of equivalence classes $v(R^{\times})=\{v(x)\mid x\in R^{\times}\}$ with  multiplication on $v(R^{\times})$ defined to be $v(x)v(y)=v(xy)$.   We endow $v(R^{\times})$ with the order
\begin{center}
$v(x)<v(y)$ if   $(\forall n\in \mathbb N)[\, n |x|<|y|\, ]$.
\end{center}
By convention, we let ${v(0)}< v(R^{\times})$.

\end{defn}

Under the given operation and ordering, $v(R^{\times})$ is an ordered Abelian group with identity $v(1)$.  Moreover, the map $x\mapsto v(x)$ is a {\em valuation}, i.e. it satisfies the axioms
$v(xy)=v(x)v(y)$ and   $v(x+y)\leq \mbox{ max }  \{ v(x), v(y)\} $.

 If $R$ is a real closed field, then the value group $v(R^{\times})$ is divisible \cite[Theorem 4.3.7]{EP}.
  An Abelian group $(G, \cdot)$ is {\em{divisible}} if for all $g\in  G$ and $0\not=n\in \mathbb N$, $g^{\frac{1}{n}}\in G$.  Note that a divisible Abelian group $(G, \cdot)$ is a     $\mathbb{Q}$-vector space when scalar multiplication by $q\in \mathbb{Q}$ is defined to be $g^q$.  This observation motivates the following definition.

\begin{defn} [Generating set]
Let $(G, \cdot)$ is a divisible Abelian group.  We say that $B$ is a
\emph{generating set} if
 each element of $G$ can be expressed as a finite product of rational powers of elements of $B$.   We denote the Abelian Group generated by a set $B\subset R$ by $\divhull{B}$.

\end{defn}

\begin{defn}  [Value group section]
A {\em value group section} is the image of an embedding of ordered  groups $t:v(R^{\times}) \hookrightarrow R^{>0}$ such that $v(t(g))=g$ for all $g\in v(R^{\times})$.
\end{defn}

If $R$ is  real closed field, there are subgroups of $(R^{>0},\cdot)$ that are value group sections (see  \cite[Theorem 8]{Kap}).  Note that we use the term ``value group section'' to refer to the image of the described embedding, not the embedding itself.
In \cite{KL}, it is shown that for a countable real closed field $R$, there is a value group section $G$ that is
$\Delta^0_2(R)$.  Moreover, this is sharp.  There is a computable real closed field $R$ such that the halting set $K$ is computable relative to every value group section.









\begin{defn} [Valuation ring]

The \emph{valuation ring} is the ordered ring \[\mathcal O_v:= \{ x\in R: v(x)\leq 1\}\] of  finite elements.

\end{defn}

The valuation ring has a unique maximal ideal \[\mathcal M_v:=\{ x\in R: v(x)<1\}\] of  infinitesimal elements.

\begin{defn} [Residue field]

The \emph{residue field} is the quotient $\mathcal O_v / \mathcal M_v$.

\end{defn}

The residue field $k$ is an ordered field under the order induced by $R$.  It is Archimedean, so it is isomorphic to a subfield of $\mathbb R$. We denote the residue of $x\in \mathcal O_v$ by $\overline{x}.$

\begin{defn} [Residue field section]
A {\em residue field section} is the image of an embedding of ordered  fields $\iota :k \hookrightarrow R$ such that $\overline{\iota (c)}=c$ for all $c\in k$.
\end{defn}

 If $R$ is a real closed field, then  $k$ is real closed \cite[Theorem 4.3.7]{EP} and residue field sections exist  \cite[Theorem 8]{Kap}.
 To construct a residue field section, we look for a maximal real closed Archimedean subfield.  In \cite{KL}, the second and fourth authors proved the following result on the complexity of residue field sections.  

\begin{prop}
\label{residue}

For a countable real closed field $R$, there is a residue field section that is
$\Pi^0_2(R)$.

\end{prop}

Proposition \ref{residue} is sharp in the following sense.

\begin{prop}

There is a computable real closed field with no $\Sigma^0_2$ residue field section.

\end{prop}

 We let $k(G)$ denote the field generated by $k\cup G$.


\begin{defn} [$k((G))$]\
Let $k$ be  an Archimedean ordered field and $G$ an ordered Abelian  group.
\begin{enumerate}

\item   The field
$k((G))$ of generalized series is the set of formal sums $s = \sum_{g\in G} a_g g$ with $a_g\in k$ and $\mbox{Supp}(s):=\{g\in G: s_g\not=0\}$ is an anti-wellordered subset of $G$.

\item  The \emph{length} of $s$ is the order type of $Supp(s)$ under the reverse ordering.  
Later, we may write $s = \sum_{i<\alpha} a_i g_i$, where $g_i\in G$ with $g_i > g_j$ for $i < j < \alpha$, and $a_i\in k^{\times}$.  Under this notation, the length of $s$ is $\alpha$.  


\item  For $s = \sum_{g\in S} a_g g$ and $t = \sum_{g\in T} b_g g$ in $k((G))$ where $Supp(s)\subset S$ and $Supp(t)\subset T$, the \emph{sum} $s+t$ and the \emph{product} $s\cdot t$ are defined as for ordinary power series.

\begin{enumerate}

\item  In $s+t$, the coefficient of $g$ is $a_g+b_g$.

\item  In $s\cdot t$, the coefficient of $g$ is the sum of the products
$a_{g'}b_{g''}$, where $g = g'\cdot g''$.

\end{enumerate}
%
\item $k((G))$ is ordered anti-lexicographically  by setting $s>0$ if $a_g>0$ where $g=:\mbox{max(Supp(s))}$.

\end{enumerate}

\end{defn}
 For a proof that $k((G))$ is a totally ordered field, see \cite[Chapter VIII, Theorem 10]{F}.
 If $k$ is real closed and $(G, \cdot)$ is an ordered divisible Abelian group, then  $k((G))$ is real closed by \cite[Theorem 4.3.7]{EP}.
 The field $k((G))$ carries a canonical valuation $v: k((G))^{\times}\rightarrow G$, defined by $s\mapsto \mbox{ max(supp}(s))$, with value group $G$.
Given a subset $X\subset G$, we set
\begin{equation*}
k((X))=\{s\in k((G))\mid Supp(s)\subset X\}.
\end{equation*}
We let $G^{\le 1}=\{g\in G\mid g\le 1\}$, and similarly define $G^{<1}$ and $G^{>1}$.
The valuation ring is the ring of finite elements $k((G^{\leq 1}))$, its valuation ideal is the ideal of infinitesimals $k((G^{< 1}))$, and  the  residue field is $k$. The canonical additive complement to the valuation ring is  $k((G^{> 1}))$, the ring of purely infinite series. The group of positive units of $k((G^{\leq 1}))$ is denoted by $\mathcal U_v^{>0}$, and consists of series $s$ in the valuation ring with  coefficient $a_g>0$ for $g=1$. In this setting the following decompositions of the additive and multiplicative groups of $k((G))$ will be useful
$$ (k((G)), +)=  k((G^{\leq 1}))\oplus k((G^{> 1})) \mbox{ and  } (k((G))^{>0}, \cdot )= \mathcal U_v^{>0}\cdot  G  .$$

\subsection{Truncation-closed embeddings}

\begin{defn} [Truncation-closed subfield]

Let $F$ be a subfield of $k((G))$.  We say that $F$ is truncation closed if whenever $s = \sum_{g\in G} a_g g \in F$ and $h\in G$, the restriction $s_{<h}=\sum_{g<h} a_g g$ also belongs to $F$.

\end{defn}

Mourgues and Ressayre \cite{MR} showed that every real closed field has an integer part in the following way.  In \cite[Lemma 3.2]{MR}, Mourgues and Ressayre observed the following.

\begin{prop} [Mourgues and Ressayre]

If $F$ is a truncation closed subfield of $k((G))$ and $Z_F$ consists of the elements of the form $t+z$, where $t\in F\cap k((G^{>1}))$ and $z\in\mathbb{Z}$, then $Z_F$ is an integer part for $F$.

\end{prop}

\begin{proof}

If $s\in F$, we have $s = t + t'$, where $t\in k((G^{>1}))$ and $t'\in k((G^{\leq 1}))$.  Take $z\in\mathbb{Z}$ such that $z\leq t' < z+1$.  Then $t+z\leq s < t+z+1$.

\end{proof}

In \cite[Corollary 4.2]{MR}, they showed the following restatement of Theorem A.

\begin{thma'}[Mourgues and Ressayre]

Let $R$ be a real closed field with value group $G$ and residue
field $k$. Then there is an order (valuation) preserving isomorphism
$\delta$ from $R$ onto a truncation closed subfield $F$ of $k((G))$.
Thus $\delta^{-1}(Z_F)$ is an integer part for $R$.
\end{thma'}

We refer to $\delta$ as a ``development function'' $\delta$.

\subsection{Exponential integer parts}
In \cite{R}  Ressayre imposes a further condition on   $\delta$ which ensures that the truncation integer part is also closed under exponentiation. The following is a rephrasing of  Theorem B and of  \cite[Theorem 4]{R}.

\begin{thmb'} Let $(R,2^x)$ be an exponential real closed field. Fix a residue field section $k\subset R$. Then there is a value group section $G\subset R^{>0}$ and a truncation closed embedding $\delta:R \hookrightarrow k((G))$ fixing $k$ and $G$, and  such that
\begin{equation}
\label{crucial} \delta (\log(G)) = \delta(R)\cap k((G^{>1})).
\end{equation}
\end{thmb'}

We argue that if $\delta$ satisfies condition (\ref{crucial}) then the truncation integer part $\delta^{-1}(Z_F)$ is an exponential integer part of $R$. The exponential function $2^x$ defined on $R$ induces an exponential function on $F=\delta (R)$ by setting: $2^y=\delta (2^x)$ where $y=\delta (x)$ for $x\in R$.
So, it suffices to show the following lemma, which appears in \cite[Proposition 5.2]{BKK}.

\begin{lem}\label{L:expIP}
 $Z_F$ is an exponential integer part of $F$ with respect to the induced exponential function.
\end{lem}
\begin{proof}

Let $z\in Z_F$ and $z>0$, then $z= a+ y$ where $y\in   F\cap k((G^{>1}))$ and $a\in \mathbb Z$. We compute $2^z=2^a2^y$.  If $y=0$ then $a>0$, so $2^z\in \mathbb N\subset Z_F$. If $y\not= 0$ then $y>0$, and $2^y>1$. We now show that $2^y\in G$. By (\ref{crucial}) $y=\delta (\log (g))$ for some $g\in G$. Then $2^y=\delta (2^{\log g})=\delta (g)=g$, as required. Therefore, $2^y\in G^{>1}$, and so $2^z=2^a2^y$ belongs to  $k((G^{>1}))$, and also to $F=\delta (R) $. So, $2^z\in F\cap k((G^{>1}))\subset Z_F$.

\end{proof}

\section{Development Triples}\label{S:MRconst}

Mourgues and Ressayre prove Theorem A$'$ by showing how to extend a partial embedding $\phi$ from a subfield $A$ of $R$ onto a
truncation closed subfield $F$ of $k((G))$ to be defined to some $r\in R-A$ while preserving truncation closure.
\begin{defn} [Development triple]

Suppose $R$ is a real closed field, with residue field $k$.  We say that $(A,H,\phi)$ is a {\em development triple with respect to  $R$ and $k$} if

\begin{enumerate}

\item $A$ is a real closed subfield of $R$ containing $k$,

\item  $H\subset A$ is a value group section for $A$, and

\item  $\phi$ is a order preserving  isomorphism from $A$ onto a truncation closed subfield of $\ek{H}$ such that $\phi\upharpoonright k(H)$ is the identity.

\end{enumerate}

\end{defn}

\noindent
\textbf{Notation}.  We write $(A', H', \phi')\supseteq(A,H, \phi)$, if $A'\supseteq A$, $H'\supseteq H$, and
$\phi'\supseteq \phi$.









Given a development triple $(A, H, \phi)$ and an element $r\in R-A$, we want to obtain a development triple $(A', H', \phi')\supset (A, H, \phi)$ with $r\in A'$.  We use the following definitions to describe $\phi'(r)$.

\begin{defn} Let $\alpha$ be an ordinal.
The {\em development of $r\in R$ over $(A, H, \phi)$ of length $\alpha$} is an element $t_\alpha\in\ek{H}$  satisfying:

\begin{itemize}
\item $t_0= 0$ if $\alpha=0$, and otherwise,
\item $t_\alpha={\sum_{i<\alpha}a_i{g_i}}$ where
\begin{equation} (\forall \beta<\alpha)(\exists \hat{r}_\beta\in A)[ \phi(\hat{r}_\beta)={\sum_{i<\beta}a_i{g_i}}\ \&\  {g_\beta}=v(r- \hat{r}_\beta)\in G ]
\end{equation}
\end{itemize}
\end{defn}

It is straightforward to prove the next lemma.

\begin{lem}\label{L:devuniq} Let $(A, H, \phi)$ be a development triple, $r\in R$, and $\alpha$ an ordinal for which $t_\gamma$ exists.  Then,
\begin{enumerate}
\item  $t_\gamma$ is unique and, for all $\beta\le \gamma$, $t_\beta=(t_\gamma)_{<\beta}$.

\item There is a development $t_\alpha$ of $r$ over $(A, H, \phi)$ of maximum length $\alpha$.
\end{enumerate}
\end{lem}
Lemma \ref{L:devuniq}  allows us to make the following definition.

\begin{defn}
The {\em maximum development of $r$ over $(A, H, \phi)$} is the unique development of $r$ over $(A, H, \phi)$ of maximum length $\alpha$.
\end{defn}

\begin{obs}\label{O:GammainH} Let $(A, H, \phi)$ be a development triple.  Let $t_\alpha$ be the maximum development of $r\in R-A$ over $(A, H, \phi)$.
 Let \mbox{$\Gamma(r)=\{v(x-r)\mid x\in A\}$.}  The following statements are equivalent:
(i) $\Gamma(r)\subset H$,
(ii) $\Gamma(r)$ has no least element,  and
(iii) $t_\alpha\not\in\phi(A)$.

Also, if  $t_\alpha\not\in\phi(A)$, then $\alpha$ is a limit ordinal or 0.
  \end{obs}
  \noindent
  Since these observations will not be used in the remainder of the paper, we omit the proof.
We now restate the key lemma of Theorem A by Mourgues and Ressayre \cite{MR} in the language of development triples.

\begin{thm}[Mourgues-Ressayre]\label{T:MRext}
Suppose $(A, H, \phi)$ is a development triple with respect to a real closed field $R$ and $r\in R-A$. There is a development triple $(A', H', \phi')$ extending  $(A, H, \phi)$ such that   $r\in A'$.  Moreover, if the maximum  development of $r$ over $(A, H, \phi)$ is  $t_\alpha\in\ek{H}$, then \mbox{$\phi'(r)_{<\alpha}=t_\alpha$.}
\end{thm}

We present Mourgues and Ressayre's construction in the framework of development triples, because we will need development triples with additional properties to examine the complexity of the exponential case, which is our main goal.

We use the following theorem to construct the development triple $(A', H', \phi')$ extending $(A, H, \phi)$.
\begin{thm}\label{T:cutsextend}  Let $A\subset A'$ and $B\subset B'$ be real closed fields such that there is an order preserving isomorphism  $\phi$ from $A$ onto $B$.  If we have $a\in A'-A$ and $b\in B'-B$ such that
\begin{equation}
(\forall x\in A)[x<a \iff \phi(x)<b],
\end{equation}
then there is a unique order preserving isomorphism $\phi'\supset\phi$ from \mbox{$RC(A\cup\{a\})$} onto $RC(B\cup\{b\})$ with $\phi'(a)=b.$
\end{thm}

We have two cases: the  immediate transcendental case where $t\not\in\phi(A)$ and the value transcendental case where $t\in\phi(A)$, as seen in \cite{Kap}.   Note there is no residue transcendental case because $k$ is a residue field section for $R$, not just $A$.
In both cases, we choose $\phi'(r)$ so that the cut of $r$ over $A$ is the same as the cut of $\phi'(r)$ over $\phi(A)$.

\begin{lem} \label{P:cut}
Let $(A, H, \phi)$ be a development triple with respect to $R$ and let $r\in R-A$.  Suppose $r$ has  development $t=t_\alpha$ over $(A, H, \phi)$.  If  $t\not\in\phi(A)$, then  for all $x\in A$,
\begin{equation}
x< r \text{ (in $R$)} \iff \phi(x)< t \text{ (in  $\ek{H}$)}
\end{equation}
Similarly, if there is some $r'\in A$ such that $\phi(r')=t$, then for all $x\in A$,
\begin{equation}
x< r \text{ (in $R$)} \iff \phi(x)< t+\epsilon g \text{ (in  $\ek{H'}$)}
\end{equation}
where  $g=|r-r'|$ and $\epsilon=\pm 1$ such that $g=|r-r'|=\epsilon (r-r')$ and $H'=\divhull{H\cup\{g\}}$.
\end{lem}
The proof of Lemma \ref{P:cut} can be found in  Lemma 3.3, p.\ 191 of \cite{VMM} (see also  Theorem 6.2, p.\ 85 of \cite{KKMZ}).
%
Lemma \ref{P:cut} can also be proved using Observation $\ref{L:cofinal}$, whose proof we omit.
\begin{obs}\label{L:cofinal}
For  an ordered set $B$,  $b\in B$, and $C\subset B$, let $C_{<b}=\{c\in C\mid c<b\}$ and let $C_{>b}=\{c\in C\mid c>b\}$.

Let $(A, H, \phi)$ be a development triple with respect to $R$ and let $r\in R-A$.  Suppose $r$ has  development $t=t_\alpha$ over $(A, H, \phi)$.
  Let $\hat{A}=\{\hat{r}_\beta\mid \beta<\gamma\}$.  Let $t'=t$ if $t\not\in\phi(A)$ and let $t'= t+\epsilon g $ as in Lemma \ref{P:cut} if $t\in\phi(A)$.
At least one of the following two statements holds.
\begin{enumerate}
\item The set $\hat{A}_{<r}$
is cofinal in $A_{<r}$, and
$\phi(\hat{A}_{<r})$
is cofinal in $\phi(A)_{<t'}$.

 \item The set $\hat{A}_{>r}$
  is coinitial in $A_{>r}$, and
$\phi(\hat{A}_{>r})$
 is coinitial in $\phi(A)_{>t'}$.
\end{enumerate}
\end{obs}

Let $t'$ be the development $t$ of $r\in R-A$ over $(A, H, \phi)$ if $t\not\in\phi(A)$, and let $t'=t+\epsilon g$ as in Lemma  \ref{P:cut} if $t\in\phi(A)$.
By Lemma  \ref{P:cut} and Theorem \ref{T:cutsextend},  there is a unique order preserving isomorphism $\phi'\supset \phi$ from $RC(A\cup \{r\})$ onto $RC(\phi(A)\cup \{t'\})$ with $\phi'(r)=t'$.  Moreover, by the definition of development, the proper truncations of $t'$ are all in $\phi(A)$.   The following lemma says that the range of $\phi'$ is truncation closed.

\begin{lem} [Mourgues-Ressayre]\label{Development}
\label{development lemma}

Let $F$ be a truncation closed subfield of $k({G})$, and let $t\in k((G)) - F$, where all proper initial segments of
$t$ are in $F$.  Then the real closure of $F(t)$ is also a truncation closed subfield of $k((G))$.
\end{lem}

Thus, in both the  immediate transcendental case and the value transcendental case, we  have defined a development triple $(A'=RC(A\cup \{r\}), H', \phi')$ extending  $(A, H, \phi)$ with $\phi'(r)=t'$.

 We have proven Theorem \ref{T:MRext}.

\begin{obs}\label{O:immedtrans}
Let $(A', H', \phi')$ be the development triple extending  $(A, H, \phi)$ such that   $r\in A'$ for $r\in R-A$ constructed in Theorem \ref{T:MRext}.
In the immediate transcendental case where $t\not\in\phi(A)$, $H'=H$, whereas in the value transcendental case where $t\in\phi(A)$, $H$ is a proper subgroup of the group $H'$ given in Lemma $\ref{P:cut}$.
\end{obs}

We will use the next notion extensively in the section on exponential case.

\begin{defn}
\begin{enumerate}
Let  $(A', H', \phi')$ and  $(A, H, \phi)$ be development triples.
\item    The triple  $(A', H', \phi')$
 is a {\em value group preserving extension} of $(A, H, \phi)$ if $(A', H', \phi')$ extends  $(A, H, \phi)$ and $H'=H$.

\item  A triple  $(A, H, \phi)$ is {\em maximal} if $(A, H, \phi)$ admits no proper value group preserving extension.
\end{enumerate}
\end{defn}

The next observation follows immediately from Observation \ref{O:GammainH}.  We will use the equivalence of the first two statements often.

\begin{lem}\label{O:maximality} Let  $(A, H, \phi)$ be a development triple.
The following are equivalent.
\begin{enumerate}
\item
$(A, H, \phi)$  is maximal.
\item For all $r\in R-A$, the development of $r$ over $(A, H, \phi)$ is in $\phi(A)$.
\item
 $(\forall r\in R-A)[\, \Gamma(r)\not\subset H].$

\item For all $r\in R-A$, $H$ is not a value group section for $RC(A\cup\{r\})$.

\end{enumerate}
\end{lem}

\begin{proof}
($1\implies 2$) Suppose $(A, H, \phi)$ is maximal.  Given $r\in R-A$.  Let $t$ be the development of $r$ over $(A, H, \phi)$.  If $t\not\in\phi(A)$,  by Theorem
\ref{T:MRext} and Observation \ref{O:immedtrans}, there is a development triple $(RC(A\cup\{r\}), H, \phi')$ with $\phi'(r)=t$ properly extending $(A, H, \phi)$, contradicting the maximality of $(A, H, \phi)$.

\noindent
($2\implies 3$)
If there is an $r\in R-A$ with $\Gamma(r)\subset H$, we have that the development $t$ of $r$ over $(A, H, \phi)$ satisfies
$t\not\in\phi(A)$ by Observation \ref{O:GammainH}.

\noindent
($3\implies 4$)
Let $r\in R-A$.  Since $\Gamma(r)\not\subset H$, there is an $r'\in A$ such that $v(r-r')\not\in H$.
Since $r-r'\in RC(A\cup\{r\})$, $H$ is not a value group for $ RC(A\cup\{r\})$.

\noindent
($4\implies 1$)  Suppose there is some $(A', H', \phi')$ extending  $(A, H, \phi)$ and $H'\not=H$.
Let $r\in A'-A$ with $v(r)\in H'-H$.
Then, $H$ is not a value group section for $RC(A\cup\{r\})$.
\end{proof}

 Note that $(k, \{1\}, id)$ is a maximal development triple as is any triple of the
 form $(R, G, \delta)$ with respect to a real closed field $R$.

\begin{lem}\label{MRcan} Given a real closed field $R$, a residue field section $k$,
and a well ordering of $R=(r_i)_{i<\lambda}$, there is  a canonical development triple
 $(R, G, \delta)$  with respect to $R$ and $k$.
\end{lem}
We construct $(R, G, \delta)$ from a chain of development triples
$(R_i,G_i,\delta_i)_{i<\lambda}$.
Once $R_i$ is defined for $i<\lambda$, we let  ${m(i)}<\lambda$ be the least ordinal such that $r_{m(i)}\in R-R_i$.
We define $(R_0,G_0,\delta_0)$ to be $(k,\{1\},id)$.
 Let $j<\lambda$.  Given development triples $(R_i, G_i, \delta_i)_{i<j}$, we define $(R_j, G_j, \delta_j)$ by induction as follows.
If $j$ is a limit ordinal, we let $R_j=\cup_{i<j} R_i$, $G_j=\cup_{i<j} G_i$, and $\delta_j=\cup_{i<j} \delta_i$.
Note that $(R_j, G_j, \delta_j)$ is a development triple since it is a union of  a chain of development triples.
If $j=l+1$ is a successor ordinal,  we  take $(R_j, G_j, \delta_j)\supset(R_l, G_l, \delta_l)$ so that $r_{m(l)}\in R_j$ using Theorem \ref{T:MRext}.

Similarly, one prove that there is a canonical maximal development triple  extension by using Theorem \ref{T:MRext} and Observation \ref{O:immedtrans} and checking whether some $r\in R$ can be added to a triple using Lemma \ref{O:maximality}.
\begin{lem}
\label{Maximal}

Let  $(A,H,\phi)$ be a development triple with respect to $R$ and $k$. Given a well ordering of $R=(r_i)_{i<\lambda}$, there is a canonical development triple $(A,H,\phi')$ extending  $(A,H,\phi)$ that is maximal.

\end{lem}






\section{Exponential integer parts}\label{S:Rconst}

To show that every real closed exponential field has an exponential integer part, Ressayre lets the value group section do most of the work.  Below, we define a special kind of development triple, with the added features we want for the group.

\begin{defn} [Dyadic development triple]

Let $R$ be a real closed exponential field, and let $k$ be a residue field section.  Let $(A, H, \phi)$ be a development triple with respect to $R$ and $k$.  Then $(A,H,\phi)$ is a \emph{dyadic development triple} for $R$ and $k$ if
\begin{equation}
\phi(\log H)=\phi(A)\cap \ek{H^{>1}}.
\end{equation}

Equivalently, $(A, H, \phi)$ is dyadic if
\begin{enumerate}

\item  for all $r\in H$, $\log r\in A$ and $\phi(\log r)\in \ek{H^{>1}}$, and
\item  for all $r\in A$, if $\phi(r)\in\ek{H^{>1}}$, then $2^r\in H$.
\end{enumerate}
\end{defn}

Rephrasing  Lemma \ref{L:expIP} in terms of this terminology, if $R$ has a dyadic triple $(R, G, \delta)$, then $R$ has an exponential integer part.  So,  proving Theorem B$'$ is equivalent to showing that every real closed exponential field $R$ with a residue field section $k$ has a dyadic triple $(R, G, \delta)$ with respect to $R$ and $k$.












%

\subsection{Extending dyadic triples}

Most of the work in the proof of Theorem B$'$ is showing how to extend one dyadic and maximal triple to another such  triple.


\begin{prop} [Main Technical Lemma]
\label{main technical} Suppose
$(A,H,\phi)$ is a dyadic and maximal triple, and $y\in R - A$.  Then there is a dyadic and maximal  triple $(A',H',\phi')\supseteq (A,H,\phi)$ such that $y\in A$.
\end{prop}
\begin{proof}

Without loss of generality, we may suppose that $y$ is positive, $v(y)>1$, and $v(y)\not\in H$.  We can take $v(y)\not\in H$ by Lemma \ref{O:maximality} (3). We may further suppose that $y>0$, since otherwise we could replace $y$ by $-y$, and we may suppose that $v(y)>1$, since otherwise we could replace $y$ by $y^{-1}$.

We will obtain the required dyadic triple $(A',H',\phi')$ as the union of a chain of maximal development triples
$(B_i,H_i,\phi_i)$ with the following features.

\begin{enumerate}

\item  $H_0\supseteq H$ is a value group section for $RC(\{\log_i(y)\mid i\in\omega\})$ where $\log_0 y=y$ and $\log_{i+1} y=\log_i(\log y)$  for all $i\in \omega$.
\item If $r\in H_i$, then $\log(r)\in B_i$ and  $\phi_i(\log(r))\in \ek{H_i^{>1}}$.
\item  If $r\in B_i$ and $\phi_i(r)\in\ek{H_i^{>1}}$, then $2^r\in H_{i+1}$.

\end{enumerate}

We begin by defining the development triple  $(B_0, H_0,\phi_0)$.
We first define a sequence $(y_i)_{i\in\omega}$ of elements in $R$ and describe some of their properties.

\begin{lem} \label{C:yinf}
  Let $(A, H, \phi)$ be a dyadic and maximal triple. Let $y=y_0\in R-A$
have the following properties for $i=0$:
\begin{equation}\label{E:v(y_i)}
y_i>0 \ \&\ v(y_i)>1\ \&\ v(y_i)\not\in H.
\end{equation}
  Given $y_i$, let $p_{i+1}$ be the development of $\log y_i$ over $(A, H, \phi)$.  We inductively assume that $y_i$ satisfies (\ref{E:v(y_i)}).  Then,
\begin{enumerate}

\item  $(\exists r'_{i+1}\in A)[\phi(r'_{i+1})=p_{i+1}]$,
\item $y_{i+1} := |\log(y_i) - r'_{i+1}|$ satisfies (\ref{E:v(y_i)}), and
\item $p_{i+1}\in\ek{H^{>1}}$.
\end{enumerate}
\end{lem}

\begin{proof} Suppose inductively that
 $y_i$ satisfies (\ref{E:v(y_i)}).  Let $p_{i+1}$ be the development  of
$\log(y_i)$ over $(A, H, \phi)$.  Since $(A, H, \phi)$ is maximal, there exists some $r'_{i+1}\in A$ such that $\phi(r'_{i+1})=p_{i+1}$.  By definition of (maximum) development, we have  that $v(y_{i+1})\not\in H$, and, in particular, $v(y_{i+1})\not=1$.
Suppose for a contradiction that $v(y_{i+1})<1$  or $y_{i+1}=0$.  We have $\log(y_i) = r'_{i+1} + \pm y_{i+1}$, and $r_{i+1}' = s +  s' $, where $\phi(s)\in\ek{H^{>1}}$ is the truncation of $\phi(r_{i+1}' )$  so that $\phi(s')\in\ek{H^{\le 1}}$.  So,  $y_i = 2^{s}2^{s'}2^{\pm y_{i+1}}$.  Since $v(s')\le 1$,  we have $2^{s'}$ equals some $c$ with $v(c)=1$.  If $v(y_{i+1})<1$  or $y_{i+1}=0$, then $2^{\pm y_{i+1}} = (1 + d)$, where $d$ is $0$ or $v(d)<1$.  Since $(A,H,\phi)$ is dyadic and $\phi(s)\in\ek{H^{>1}}$, we have $2^{s}\in H$.  Then, $v(y_i)=2^{s}$ , contradicting our assumption that $v(y_i)\not\in H$.
So, $y_{i+1}\not=0$ and $v(y_{i+1})>1$.   Since  $v(y_{i+1})<v(g)$ for all $g\in Supp(\phi(r'_{i+1}))$, we see that $\phi(r'_{i+1})=p_{i+1}\in\ek{H^{>1}}$.

\end{proof}

\begin{lem}
\label{y_iincr}

For all $i,n \in\omega$, $(y_{i+1})^n < y_i$.    Hence, $v(y_i)\not= v(y_j)$ for $i\not=j$.

\end{lem}

\begin{proof}
From the definition of $y_{i+1}$, we see that $y_{i+1} < \log(y_i)$, so
$y_{i+1}^n < \log(y_i)^n$.  Since $v(y_i)>1$,
$\log(y_i)^n < 2^{\log(y_i)}=y_i$ by property (5) of  Definition \ref{D:EXP}.

\end{proof}

  Let $H_{0,n}=\divhull{H\cup\{y_i\mid i\in\omega\}}$.  Let $H_0=\cup_{n\in\omega} H_{0,n}$.

\begin{lem}
For each $n$, $v(y_n)\not\in H_{0,n}$.   Hence, $H_0\supset H$ is a value group section for $RC(A\cup H_0)$.

\end{lem}

\begin{proof}

The statement is clear for $n = 0$.
We show the statement for $n+1$.  We assume for a contradiction that $v(y_{n+1})\in H_{0, n}$, i.e.,
$y_{n+1}=cgy_0^{q_0}\cdot \ldots\cdot y_n^{q_n}$, where
$c\in R$, $v(c)=1$, $g\in G$, and $q_i\in \mathbb{Q}$.  Taking logs, we obtain
\begin{equation}
\log(y_{n+1}) = \log(c)+\log(g) + q_0\log(y_0)+\ldots+q_n\log(y_n).
\end{equation}
Recall that, by definition, $\log(y_i) = r'_{i+1}+\epsilon_{i+1} y_{i+1}$, where $\phi(r'_{i+1})$ is the development of $\log(y_i)$ over $(A,H,\phi)$ and $\epsilon_{i+1} = \pm 1$.    Then, by subsitution and
rearranging terms, we have  that $ \epsilon_{n+2}y_{n+2}$ equals
\begin{equation*}
  \log(c) + [\log(g) + q_0r'_1+\ldots + q_nr'_{n+1}-r'_{n+2}]+[q_0\epsilon_1y_1 + \ldots + q_n\epsilon_{n+1}y_{n+1}]
\end{equation*}
We have $v(\log c)=1$, $v(\log(g) + q_0r'_1+\ldots + q_nr'_{n+1}-r'_{n+2})\in H^{>1}$, and $v(q_0\epsilon_1y_1 + \ldots + q_n\epsilon_{n+1}y_{n+1})=v(y_1)$ by Lemmas \ref{C:yinf} and \ref{y_iincr}.  Thus,  $v(y_{n+2})$ is either in $H$ or equals $v(y_1)$, contradicting either Lemma   \ref{C:yinf} or Lemma  \ref{y_iincr}.
\end{proof}

\begin{lem}\label{O:logh}
If $h\in H_0$, then  the development of $\log h$ over $(A, H, \phi)$ is in $\ek{H_0^{>1}}$.
\end{lem}
\begin{proof}
Let  $h\in H_0$ with $h=g\prod_{i=0}^n y_{l_i}^{q^i}$.  So,
\begin{equation}
\log h=\log g+\sum_{i=0}^n q_i(r'_{i+1}+\pm y_{i+1})
\end{equation}  where $q_i\in \mathbb{Q}$.   The developments of $\log g$  and $r'_{i+1}$ are in  $\ek{H^{>1}}$ since $(A, H, \phi)$ is dyadic and by construction.  Since $v(y_{i+1})>1$, $\log(h)$ has a development in $\ek{H_0^{>1}}$.
\end{proof}

 By Theorem \ref{T:MRext} and Lemma \ref{Maximal}, we obtain $B_0$ and $\phi_0$ such that  $(B_0, H_0, \phi_0)$ is maximal and extends $(A, H, \phi)$.

We define $H_1=\divhull{H_0\cup\{2^r\mid r\in B_0 \ \&\ \ \phi_0(r)\in \ek{H_0^{>1}}\}}$.  As was the case for $H_0$,   for all $h\in H_1$, $\phi_0(\log h)\in \ek{H_0^{>1}}$.  The next lemma ensures that $H_1$ is a value group section for $RC(B_0\cup H_1)$.







Given $(B_i,H_i,\phi_i)_{i<\alpha}$ such that $(B_i,H_i,\phi_i)$ is maximal  and $\phi_i(h)\in\ek{H_i^{>1}}$ for all $h\in H_i$, we define the triple $(B_\alpha,H_\alpha, \phi_\alpha)$.  If $\alpha=j+1$,
 we let $H_{j+1}=\divhull{H_j\cup\{2^r\mid r\in B_j \ \&\ \ \phi_j(r)\in \ek{H_j^{>1}}\}}$.   By Theorem \ref{T:MRext} and Lemma \ref{Maximal}, we obtain $B_{j+1}$ and $\phi_{j+1}$ such that $(B_{j+1}, H_{j+1}, \phi_{j+1})$ is maximal and extends  $(B_j,H_j,\phi_j)$.

For $\alpha$ a limit ordinal, we let $H_\alpha = \cup_{i<\alpha} H_{i}$.  If $(\cup_{i<\alpha}  B_{i}, H_\alpha, \cup_{i<\alpha}  \phi_{i})$ is not maximal, then we use Lemma \ref{Maximal} to find a maximal triple $(B_\alpha, H_\alpha, \phi_\alpha)$ extending  $(\cup_{i<\alpha}  B_{i}, H_\alpha, \cup_{i<\alpha}  \phi_{i})$.
If $(\cup_{i<\alpha}  B_{i}, H_\alpha, \cup_{i<\alpha}  \phi_{i})$ is  maximal, then we set $B_\alpha=\cup_{i<\alpha}  B_{i}$ and $\phi_\alpha=\cup_{i<\alpha}  \phi_{i}$.  In this case,   $(B_\alpha, H_\alpha, \phi_\alpha)$ is the desired dyadic triple $(A',H',\phi')$.

The  analogue of the proof of Lemma \ref{O:logh} shows the following.

\begin{lem}\label{L:Halphalog}

For all $\alpha$,  if $h\in H_\alpha$, then $\phi_\alpha(\log h)\in \ek{H_\alpha^{>1}}$.

\end{lem}

The next lemma shows that $H_\alpha$ is a value group section for $RC(\cup_{i<\alpha} B_i \cup H_\alpha)$.
\begin{lem}\label{L:valimp=}

For all $\alpha$, if $h,h'\in H_\alpha$ and $v(h)=v(h')$, then $h = h'$.

\end{lem}

\begin{proof}

If $v(h)=v(h')$, then $h = ch'$, for some $c\in R^{>0}$ with $v(c)=1$.  By Lemma \ref{L:Halphalog}, we have $\phi_\alpha(\log h), \phi_\alpha(\log h')\in  \ek{H_\alpha^{>1}}$.
 Since  $\log(h) = \log(c) + \log(h')$, we must have
$\phi_\alpha(\log h)=\phi_\alpha(\log h')$  and $\log(c) = 0$, so $c = 1$.

\end{proof}

Since $R$ is a set, there exists some  limit ordinal $\lambda$ such
that $B_\lambda = \cup_{i<\lambda} B_{i}$ and
$\phi_\lambda=\cup_{i<\lambda} \phi_{i}$.  Then
$(B_\lambda,H_\lambda,\phi_\lambda)$ is a dyadic and maximal triple
extending
 $(A, H, \phi)$ and for which $y\in B_\lambda$, as required for Proposition \ref{main technical}.
\end{proof}

\begin{lem}\label{L:canRconst}
Given  a real closed exponential field $R$, a residue field section $k$,  and a well ordering of $R=(r_i)_{i<\lambda}$, there is a canonical
 dyadic triple $(R,G,\delta)$ with respect to these data.
 \end{lem}

The proof is the same as in Lemma \ref{MRcan}, except that we use the
   following corollary at limit steps in our construction.

\begin{cor}
\label{cor}
Suppose $(A, H,\phi)$ is the union of a chain of dyadic triples.  Then there is a dyadic and maximal triple
$(A',H',\phi')$ extending $(A, H,\phi)$.
\end{cor}

\begin{proof}
The triple $(A, H,\phi)$ may not be maximal.  By
Lemma \ref{Maximal}, we extend  $(A, H,\phi)$ to a maximal triple $(\hat{A}, H,\hat{\phi})$.  If $\hat{A} = R$, then $(R, H,\hat{\phi})$ is a dyadic triple.    If not, take $i<\lambda$ least such that $r_i\in R-\hat{A}$.
 By Proposition  \ref{main technical}, we can extend $(\hat{A}, H,\hat{\phi})$ to a dyadic and maximal triple $(A',H',\phi')$.

\end{proof}


%



%


\section{Recursive saturation, Barwise-Kreisel Compactness, and $\Sigma$-saturation}

For our example illustrating the complexity of Ressayre's Construction, we shall use recursive saturation and a version of Compactness for computable infinitary sentences.  We also describe a different method for producing exponential integer parts.  For this we need $\Sigma$-saturation, a kind of saturation for infinitary formulas.  Recursive saturation has already come up in connection with integer parts.  In \cite{DKS}, it was shown that a countable real closed field has an integer part satisfying Peano arithmetic if and only if the real closed field is Archimedean or recursively saturated.

\subsubsection{Recursive saturation}

Recursive saturation was defined by Barwise and Schlipf \cite{BS}.

\begin{defn} [Recursive saturation]

A structure $\mathcal{A}$ is \emph{recursively saturated} if for all tuples $\overline{a}$ in $\mathcal{A}$ and all c.e.\ sets of formulas $\Gamma(\overline{a},x)$, if every finite subset of $\Gamma(\overline{a},x)$ is satisfied in $\mathcal{A}$, then some $b\in\mathcal{A}$ satisfies all of $\Gamma(\overline{a},x)$.
\end{defn}

Countable recursively saturated structures can be expanded as follows.

\begin{thm} [Barwise-Schlipf]
\label{Barwise-Schlipf}

Let $\mathcal{A}$ be a countable recursively saturated $L$-structure.  Let $\Gamma$ be a c.e.\ set of sentences, in a language $L'\supseteq L$.  If the consequences of $\Gamma$ in the language $L$ are true in $\mathcal{A}$, then $\mathcal{A}$ can be expanded to a model of $\Gamma$.

\end{thm}








In \cite{MM} Macintyre and Marker considered the complexity of recursively saturated models.
We shall need the following result.

\begin{thm} [Macintyre-Marker]
\label{MM}

Suppose $E$ is an enumeration of a countable Scott set $\mathcal{S}$.  Let $T$ be a complete theory in $\mathcal{S}$.  Then $T$ has a recursively saturated model $\mathcal{A}$ such that $D^c(\mathcal{A})\leq_T E$.

\end{thm}

The next result may be well-known.  The proof will be obvious to anyone familiar with the proof of Theorem \ref{Barwise-Schlipf}.

\begin{prop}
\label{B-S complexity}

Suppose $\mathcal{A}$ is a countable recursively saturated structure, say with universe $\omega$, and let $\Gamma$ be a c.e.\ set of finitary sentences, in an expanded language, such that the consequences of $\Gamma$ are all true in $\mathcal{A}$.  Then $\mathcal{A}$ can be expanded to a model $\mathcal{A}'$ of $\Gamma$ such that $D^c(\mathcal{A}')$ is computable in the jump of $D^c(\mathcal{A})$.

\end{prop}

\begin{proof} [Proof Sketch]

We carry out a Henkin construction, as Barwise and Schlipf did, and we observe that the jump of
$D^c(\mathcal{A})$ is sufficient.
We make a recursive list of the sentences $\varphi(\overline{a})$ in the expanded language,
with names for the elements of $\omega$.  We also make a recursive list of the c.e.\ sets $\Gamma(\overline{a},x)$.  At each stage $s$, we have put into $D^c(\mathcal{A}')$ a c.e.\ set $\Sigma_s(\overline{a})$ of sentences involving a finite tuple of constants, such that the consequences in the language of $\mathcal{A}$ are true in $R$ of the constants $\overline{a}$.  At stage $s+1$, we consider the next sentence $\varphi(\overline{a})$.  We add $\varphi(\overline{a})$ to $\Sigma_s(\overline{a})$ if our consistency condition is satisfied, and otherwise we add the negation.  Then we consider the next c.e.\ set $\Gamma(\overline{a},x)$.  To check consistency, we see if the consequences of adding this, with some new constant $e$ for $x$, are true of $\overline{a}$.  Then we look for $b$ such that for $b = x$, the consequences are satisfied by $\overline{a},b$.

\end{proof}

\subsubsection{Compactness for infinitary logic}

Kripke-Platek set theory ($KP$) differs from $ZFC$ in that the power set axiom is dropped, and the separation and replacement axioms are restricted to formulas with bounded quantifiers.  An \emph{admissible set} is a model of $KP$ that is \emph{standard}; i.e., the epsilon relation is the usual one and the model forms a transitive set.  If $A$ is an admissible set, and $B\subseteq A$, then $B$ is \emph{$\Sigma_1$ on $A$} if it is defined by an existential formula, possibly with parameters.  A set is \emph{$A$-finite} if it is an element of $A$.  The least admissible set is $A = L_{\omega_1^{CK}}$.  In this case, a set $B\subseteq\omega$ is $\Sigma_1$ on $A$ if it is $\Pi^1_1$, and it is $A$-finite if it is hyperarithmetical.  

For a countable language $L$, there are uncountably many formulas of $L_{\omega_1\omega}$.  For a countable admissible set $A$, the \emph{admissible fragment $L_A$} consists of the $L_{\omega_1\omega}$ formulas that are elements of $A$.  In the case where $A$ is the least admissible set, the $L_A$-formulas are essentially the computable infinitary formulas.

\begin{thm} [Barwise Compactness]\label{T:BarwiseCompact}

Let $A$ be a countable admissible set, and let $L$ be an $A$-finite language.  Suppose $\Gamma$ is a set of $L_A$-sentences that is $\Sigma_1$ on $A$.  If every $A$-finite subset of $\Gamma$ has a model, then $\Gamma$ has a model.

\end{thm}

As a special case, we have the following.

\begin{thm} [Barwise-Kreisel Compactness]

Let $L$ be a computable language.  Suppose $\Gamma$ is a $\Pi^1_1$ set of computable infinitary $L$-sentences.  If every hyperarithmetical subset of $\Gamma$ has a model, then $\Gamma$ has a model.

\end{thm}

Ressayre's notion of $\Sigma_A$-saturation, defined in \cite{R0}, \cite{R1} is associated with Barwise Compactness.  We start with an admissible set $A$.  Some people omit the axiom of infinity from $KP$.  Then $L_\omega$ qualifies as an admissible set, and we get recursive saturation as a special case, where $A = L_\omega$.  Ressayre worked independently of Barwise and Schlipf, and the first version of his definition, in \cite{R0}, was actually earlier.

\begin{defn}

Suppose $A$ is a countable admissible set and let $L$ be an $A$-finite language.  An $L$-structure $\mathcal{A}$ is \emph{$\Sigma_A$-saturated} if

\begin{enumerate}

\item  for any tuple $\overline{a}$ in $\mathcal{A}$ and any set $\Gamma$ of $L_A$-formulas, with parameters
$\overline{a}$ and free variable $x$, if $\Gamma$ is $\Sigma_1$ on $A$ and every $A$-finite subset is satisfied, then the whole set is satisfied.

\item  let $I$ be an $A$-finite set, and let $\Gamma$ be a set, $\Sigma_1$ on $A$, consisting of pairs $(i,\varphi)$, where $i\in I$ and $\varphi$ is an $L_A$-sentence.  For each $i$,
 let $\Gamma_i = \{\varphi:(i,\varphi)\in\Gamma\}$.  Similarly, if $\Gamma'\subseteq\Gamma$, let $\Gamma'_i = \{\varphi:(i,\varphi)\in\Gamma'\}$.  If for each $A$-finite $\Gamma'\subseteq\Gamma$, there is some $i$ such that all sentences in $\Gamma'_i$ are true in $\mathcal{A}$, then there is some $i$ such that all sentences in $\Gamma_i$ are true in $\mathcal{A}$.

\end{enumerate}

\begin{prop}\label{P:SigmaSat}

A countable structure $\mathcal{A}$ is $\Sigma_A$-saturated iff it lives in a fattening of $A$, with no new ordinals.

\end{prop}

\end{defn}

Countable $\Sigma$-saturated models have the property of expandability.

\begin{thm} [Ressayre]\label{T:SigmaSat}

Suppose $\mathcal{A}$ is a countable $\Sigma_A$-saturated $L$-structure.  Let $L'\supseteq L$, and let $\Gamma$ be a set of $L'_A$-sentences, $\Sigma_1$ on $A$, s.t.\ the consequences of $\Gamma$, in the language $L$, are all true in $\mathcal{A}$.  Then $\mathcal{A}$ has an expansion satisfying $\Gamma$.  Moreover, we may take the expansion to be $\Sigma_A$-saturated.

\end{thm}

\subsection{Complexity of integer parts}

In \cite{KL}, the second and fourth authors studied the complexity of some parts of the construction of Mourgues and Ressayre.  In particular, they proved the following result, which we shall use later.

\begin{prop}

For a countable real closed field $R$, there is a residue field section $k$ that is $\Pi^0_2(R)$.

\end{prop}.



\section{Complexity of exponential integer parts}\label{S:compEIP}

We now turn to our main new result.  We show that there is a real closed exponential field with a residue field section and a well ordering, all arithmetical, such that Ressayre's construction is not completed in $L_{\omega_1^{CK}}$.

\subsection{Complexity of Ressayre's construction}\label{SS:compR}

We turn to the complexity of Ressayre's construction of an exponential integer part for a real closed exponential field.  Let $R$ be a countable real closed exponential field.  By Lemma \ref{L:canRconst}, given a fixed residue field section $k$ and a well ordering $\prec$ of the elements of  $R$, then Ressayre's construction of an exponential integer part is canonical.  To fix notation, let   $(R_0, G_0, \delta_0)$ be the development triple with $R_0=k$, $G_0=\{1\}$, and $\delta_0=id$.    Let $y$ be the  $\prec$-first element  of $R-k$, adjusted so that $y$ is positive and infinite.  We will focus in this section on the chain of development triples $(B_j,H_j,\delta_j)_{j<\zeta}$ leading to the first non-trivial maximal and dyadic triple $(R_1, G_1, \delta_1)$ with $y\in R_1$.  The chain  $(B_j,H_j,\phi_j)_{j<\zeta}$ is defined in Proposition \ref{main technical}  and  each element extends $(R_0, G_0, \delta_0)$.  We recall that this chain satisfies the following conditions:

\begin{enumerate}

\item  $H_0=\divhull{\{y_i = \log^i(y)\mid i\in\omega\}}$,

\item  $H_{j+1}=\divhull{H_j\cup\{2^r\mid r\in B_j \ \&\ \ \phi_j(r)\in \ek{H_j^{>1}}\}}$,

\item  for limit $j$, $H_j = \cup_{j' < j} H_{j'}$,

\item  for all $j$, $B_j$ is maximal for $H_j$, obtained by applying Lemma \ref{Maximal}

\item The length of the chain is the first limit ordinal $\zeta$ such that $\cup_{j<\zeta} B_j$ is maximal for
$H_{\zeta} = \cup_{j<\zeta} H_j$.

\end{enumerate}

Each step is sufficiently effective that the whole construction is constructible.  However, we would like to know whether the entire construction can be completed in $L_{\omega_1^{CK}}$, the hyperarithmetical universe.    There are two possible sources of complexity, of which  the ordinals required for the construction play an important role.

\begin{enumerate}

\item  An object constructed  at some step in the chain of development triples needed to arrive  at the first non-trivial maximal dyadic triple may not be hyperarithmetical. Such objects include the  lengths of  developments  in the triples.

\item  The length of the chain of development triples needed to arrive  at the first non-trivial maximal dyadic triple may not be hyperarithmetical.

\end{enumerate}

We produce an example of a hyperarithmetical real closed field for which Ressayre's construction cannot be completed in $L_{\omega_1^{CK}}$.  Let $\mathcal{C}$ be the chain of development triples leading to  the first non-trivial maximal dyadic triple for our example.  We show that even if all objects in the construction of $\mathcal{C}$, including the lengths of all developments, are hyperarithmetical, then the length of $\mathcal{C}$ is not hyperarithmetical.

\begin{thm}
\label{offcharts}

There is a low real closed exponential field $R$, with a $\Delta^0_3$ residue field section $k$ and a $\Delta^0_3$
ordering $<$ of type $\omega+\omega$, such that Ressayre's construction, even of the first non-trivial maximal and dyadic triple
$(R_1,G_1,\delta_1)$, is not completed in $L_{\omega_1^{CK}}$.

\end{thm}

To prove Theorem \ref{offcharts}, we use recursive saturation together with Barwise Compactness.
We begin by fixing  the particular real closed exponential field $R$  cited in Theorem \ref{offcharts}.

\begin{lem}

There is a recursively saturated real closed exponential field $R$ such that $D^c(R)$ is low.

\end{lem}

\begin{proof} [Proof of Lemma]

By the low basis theorem, there is a low completion $K$ of $PA$.  Let $\mathcal{S}$ be the Scott set consisting of sets representable with respect to $K$.  There is an enumeration $E$ $\mathcal{S} = Rep(K)$ such that $E\leq_T K$.  In
$\mathcal{S}$, we find a completion $T$ of the set of axioms for real closed exponential fields.
By Theorem \ref{MM}, of Macintyre and Marker \cite{MM}, there is a recursively saturated model $R$ of $T$ such that $D^c(R)\leq_T E$.  This is the real closed exponential field $R$ that we want.

\end{proof}

We now choose the residue field section $k$ for $R$ cited in Theorem \ref{offcharts}.  We do not ask that it respect the exponential function.

\begin{lem}

There is a $\Delta^0_3$ residue field section $k$ for $R$.

\end{lem}

\begin{proof}

There is a residue field section $k$ that is $\Pi^0_2(R)$.  Since $R$ is low, $k$ is $\Delta^0_3$.

\end{proof}

To prove Theorem \ref{offcharts}, we  construct a $\Delta^0_3$ well ordering $\prec$ of $R$
so that Ressayre's construction in Lemma \ref{L:canRconst} either produces a non-hyperarithmetical object in $(B_j,H_j,\delta_j)_{j<\zeta}$  or so that this chain of development triples leading up to $(R_1, G_1, \delta_1)$ has length greater $\omega_1^{CK}$, i.e.,  $\zeta\ge \omega_1^{CK}$.  We will apply  Barwise Compactness  (Theorem \ref{T:BarwiseCompact}) to particular set of set of sentences $\Gamma$.
Let $\Gamma$ consist of the following sentences.

\begin{enumerate}

\item  An infinitary sentence $\psi$ characterizing the $\omega$-models of $KP$.

An $\omega$-model has the feature that each element of the definable element $\omega$ has only finitely many elements, a fact that we can express using a computable infinitary sentence.  An $\omega$-model of $KP$ contains the hyperarithmetical sets.  In particular, we have the real closed exponential field $R$ and the residue field section $k$, with the indices we have chosen for them.

\item  A finitary sentence $\varphi_{\prec}$ saying of a new symbol $\prec$ that it is a $\Delta^0_3$ ordering of $R$ of order type $\omega+\omega$.

The sentence $\varphi_{\prec}$ states that there exists an element of $\omega$ that is a $\Delta^0_3$ index for $\prec$, and $(R, \prec)$ is isomorphic to the ordinal $\omega+\omega$.  Necessarily, the isomorphism will be $\Delta^0_3$, so it is an element of any model of $KP$.

\item  A sentence $\varphi_\alpha$, for each computable limit ordinal $\alpha$, saying that if for all $\beta<\alpha$,
the triples $(B_j,H_j,\gamma_j)$ are in $L_\alpha$ for all $j \leq \beta$, then $B_\beta\not= \cup_{j < \beta} B_j$.


Note that we identify an element of $k((H_i))$ of length $<\alpha$ with a decreasing function from $\alpha$ to $k\times H_j$.

\end{enumerate}

We must show that every hyperarithmetical subset of $\Gamma$ has a model.
For a computable ordinal $\alpha$, let $\Gamma_\alpha$ consist of $\psi$, $\varphi_{\prec}$ and $\varphi_\beta$ for
\mbox{$\beta<\alpha$.}  Each hyperarithmetical subset of $\Gamma$ is included in one of the sets $\Gamma_\alpha$.  So, to show $\Gamma$ is consistent, we must show each $\Gamma_\alpha$ is consistent for all $\alpha<\omega_1^{CK}$.  In other words, for each $\alpha<\omega_1^{CK}$,
we must show that there is a $\Delta^0_3$ ordering $\prec_\alpha$ on $R$, of order type $\omega+\omega$, so that when  Ressayre's construction is run according to the well ordering $\prec_\alpha$, if all triples $(B_j, H_j, \gamma_j)$ are in $L_\alpha$ for  $j<\alpha$, then the length of the chain of development triples $(B_j,H_j,\phi_j)$ leading to the first nontrivial maximal and dyadic triple  is greater than $\alpha$.

\subsection{Special elements}

Let $\alpha$ be a computable ordinal, and we fix a path through $\mathcal{O}$.  To show that $\Gamma_\alpha$ has a model, we use some special elements, which we name by constants $y$, $y_i$, $i\in\omega$, $c_\beta$, for $\beta<\alpha$ either $0$ or a limit ordinal, and $c_{\beta,i}$ for all $\beta < \alpha$ and all $i\in\omega$.  We first state some properties that we would like for the constants.  We define all of the constants in terms of $y$, $c_0$, and $c_\beta$ for limit $\beta<\alpha$.  Assuming that $(B_j, H_j, \gamma_j)$ are in $L_\alpha$ for  $j\le \beta$, we want $c_\beta\in B_\beta - \cup_{\gamma<\beta} B_\beta$.  To assure this, we specify a development that we want for $c_\beta$, in terms of constants $c_{\gamma+1,i}$ for $\gamma<\beta$, which we want in $H_{\gamma+1} - H_\gamma$.

In order to obtain a model of $\Gamma_\alpha$ with these constants, we give a c.e.\ set of  finitary axioms, partially describing our constants. 
Since $R$ is recursively saturated, we then apply Theorem \ref{Barwise-Schlipf}  of Barwise and Schlipf to get an expansion $R_\alpha$ satisfying these  finitary axioms.  Finally, we define a $\Delta_3^0$ well ordering $\prec_\alpha$ so that when  Ressayre's construction is run using $R_\alpha, k$, and $\prec_\alpha$ the constants have all of the desired properties.  Hence, $R_\alpha$, $k$, $\prec_\alpha$ and $L_{\omega_1^{CK}}$ witness the truth of $\Gamma_\alpha$.

\subsubsection{Descriptions of the constants}

Set $y$ to be a positive and infinite element of $R$, and  set $y_i = \log^i(y)$, so that all the $y_i$ are positive and infinite and satisfy
\[y_0 > y_1 > y_2 > y_3 \ldots \]
We will define the well ordering $\prec_\alpha$ so that $y_i\in H_0$ for all $i\in\omega$ when Ressayre's construction is run according to $\prec_\alpha$.  Furthermore, we will define $\prec_\alpha$ and ensure the constants satisfy certain properties so that they will be assigned particular developments.


We want $c_0$ to have the development $\sum_{1\le i<\omega} y_i$, which is a development in $\ek{H_0^{>1}}$ if $y_i\in H_0$ for all $i<\omega$. The description of $c_0 = c_{0,1}$ is
\begin{center}
 $y_1 < c_0  < 2y_1$, $y_2 < c_0 - y_1 < 2y_2$, etc.\\
 and \\
$c_{0, i}=c_0-\sum_{j=1}^{i-1} y_j$.
\end{center}
We defined $c_{0,i}$  so that if $c_0$ is assigned the  development $\sum_{1\le i<\omega} y_i$, then $c_{0,i}$ will have the development $\sum_{i \le j<\omega} y_j$.

\item Let $\gamma\le \alpha$ be a successor ordinal with $\gamma=\beta+1$.
We  define  $c_{\gamma, j}$ to be $c_{\beta+1,j}=2^{c_{\beta,j+1}}$ for $0<j<\omega$.


\item  Let $\gamma\le\alpha$ be a limit ordinal    where the   notation for
$\gamma$ in our fixed path through $\mathcal{O}$ gives a sequence of successor ordinals $(\gamma_i)_{i\in\omega}$ converging to $\gamma$.   
The description of $c_{\gamma}=c_{\gamma, 1}$ is
\begin{center}
$c_{\gamma_1,1} < c_\gamma < 2c_{\gamma_1,1}$, $c_{\gamma_2,2} < c_\gamma - c_{\gamma_1,1} < 2c_{\gamma_2,2}$, etc., and \\
$c_{\gamma, i}=c_\gamma-\sum_{j=1}^{i-1} c_{\gamma_j, j}$.
\end{center}
This completes our description of any element $c_{\gamma, i}$ for $\gamma\le\alpha<\omega_1^{CK}$ and $0<i\in\omega$.

Suppose $\gamma\le\alpha$ is a limit ordinal.  We want $c_{\gamma, i}$ to be assigned the development $\sum_{ i\le j<\omega} c_{\gamma_j,j}$.   In order for  $\sum_{ i\le j<\omega}  c_{\gamma_j,j}$ to be a development under  Ressayre's construction with $\prec_\alpha$, we will need to ensure that each element $c_{\gamma_j,j}$ is a member of the value group section $G_1$ and that \mbox{$c_{\gamma_j, j}>c_{\gamma_{j+1}, j+1}$} for all nonzero $j\in\omega$.  The next lemma shows that the latter condition holds.   Later, we will choose the well ordering $\prec_\alpha$ on $R$ carefully to ensure that the former condition holds as well.

\begin{lem}
The descriptions of the constants $c_{\beta, i}$ for  $\beta\le\alpha$ and  $i\in\omega$ imply that for all  $\beta\le\alpha$
\begin{equation}\label{E:y>rbeta}
y_0>c_{\beta,1}>y_1>c_{\beta,2}>y_2>c_{\beta,3}>y_3>c_{\beta,4}> \ldots
\end{equation}
\end{lem}

\begin{proof}
From the description of $c_{0, i}$, we can see that (\ref{E:y>rbeta}) holds for $\beta=0$.   Let $\gamma\le \alpha$ be a successor ordinal with $\gamma=\beta+1$.  We inductively assume that  the descriptions for the elements $c_{\beta, k}$ imply the ordering in (\ref{E:y>rbeta}).  By applying a power of $2$ to the inequalities in (\ref{E:y>rbeta}) and the definition of $c_{\gamma, i}$, we obtain the ordering
\begin{equation}
y_0>c_{\gamma,1}>y_1>c_{\gamma,2}>y_2>c_{\gamma,3}>y_3>c_{\gamma,4}> \ldots
\end{equation}

 Let $\gamma\le\alpha$ be a limit ordinal    where the   notation for
$\gamma$ in our fixed path through $\mathcal{O}$ gives a sequence of successor ordinals $(\gamma_i)_{i\in\omega}$ converging to $\gamma$.   Moreover, by induction, we have that the descriptions of the $(c_{\gamma_i, i})_{i\in\omega}$ imply that

\begin{equation}
y_0>c_{\gamma_1,1}>y_1>c_{\gamma_2,2}>y_2>c_{\gamma_3,3}>y_3>c_{\gamma_4,4}> \ldots
\end{equation}

By the description of $c_{\gamma, i}$, we have that
\begin{equation}
y_0>c_{\gamma, 1}>c_{\gamma_1,1}>y_1>c_{\gamma, 2}>c_{\gamma_2,2}>y_2>c_{\gamma, 3}>c_{\gamma_3,3}>y_3> \ldots,
\end{equation}
completing the induction.
\end{proof}

For the given computable ordinal $\alpha$, we may take the set of finitary sentences describing the constants to be computably enumerable.    Since $R$ is recursively saturated, we get an expansion $R_\alpha$ of $R$ with special elements $c_{\beta, i}\in R$ satisfying the appropriate sentences for $\beta\le\alpha$ and $i\in\omega$ by Theorem \ref{Barwise-Schlipf}.  We may take $R_\alpha$ to be $\Delta_2^0$ since $R$ is low and the oracle $\Delta^0_2$ can determine whether an element satisfies a given description of some $c_{\beta}$.

\subsection{The ordering}

We now describe a $\Delta^0_3$ well ordering $\prec_\alpha$ of $R$ such that when Ressayre's construction is run on $R$, $k$, and $<_\alpha$, if $(B_j, H_j, \gamma_j)$ are in $L_\alpha$ for  $j<\alpha$, then the development chain leading to the first nontrivial maximal and dyadic triple has length greater than $\alpha$.  Moreover, we construct $\prec_\alpha$ so that $(R, \prec_\alpha)$ has order type $\omega+\omega$.  We set $y$ to be the  $\prec_\alpha$-least element of $R$.  The special elements $c_\beta$ for $\beta\le \alpha$ will make up the remainder of the initial segment of type $\omega$, and the other elements will make up the remaining segment of type
$\omega$.  Two elements in the same $\omega$ segment are ordered according to the standard type $\omega$ ordering on their codes.  Since $R_\alpha$ is $\Delta_2^0$, we can use $\Delta_3^0$ to determine, for a given $r\in R_{\alpha}$,  whether there exists some $\beta\le\alpha$ such that $r=c_\beta$, i.e., whether $r$ should be placed in the intial $\omega$ segment or the latter.  Hence, $\prec_\alpha$ is $\Delta_3^0$.

We now run Ressayre's construction on $R$, $k$, and $\prec_\alpha$ as described.  For the remainder of the section, we let $(B_i, H_i, \phi_i)_{i<\zeta}$ be the chain of development triples in this construction  leading up to the first maximal and nontrivial dyadic triple $(R_1, G_1, \delta_1)$.  We want to show that $\zeta>\alpha$.
\subsubsection{Lemmas about the constants}

We want to show that the constants $c_\beta$ get the developments we want for them.
The following lemmas are useful.

\begin{lem}

For all $\beta\leq\alpha$, for all $h\in H_\beta^{>1}$, there is some $i$ such that $h > y_i$.

\end{lem}

\begin{proof}

We proceed by induction on $\beta$.  Since $y$ is the $\prec_\alpha$-least element of $R$, $H_0$ equals $\divhull{y_i\mid i<\omega}$.
So, $h\in H_0$ is a finite product of rational powers of the $y_i$.  Let $i$ be least such that there is a factor $y_i^{q_i}$.  Since $h\in H_0^{>1}$, $q_i$ must be positive.  Then $h > y_{i+1}$.  Suppose the statement holds for $\beta$, and $h \in H_{\beta+1}^{>1}$.  By construction, we may assume that $h = 2^r$, where $\phi_\beta(r)\in \ek{H_\beta^{>1}}$ has a positive initial coefficient.

  Say $w(r) = h' > y_i$.  Then $h > 2^{y_i} > y_i$.  Finally, suppose the statement holds for $\gamma < \beta$, where $\beta$ is a limit ordinal.  Since $H_\beta = \cup_{\gamma<\beta} H_\gamma$, the statement holds for $H_\beta$.

\end{proof}

\begin{lem}\label{L:logh}
For all $\beta\le\alpha$, if $h\in H_\beta$ and $h\not=1$, then there is some $\gamma\le \alpha$ with $\gamma=0$ or $\gamma<\beta$ such that  $\log h\in B_\gamma$ and $\delta_\gamma(\log h)\in \ek{H_\gamma^{>1}}$.
\end{lem}

\begin{proof}
We prove the lemma by induction on $\beta\le\alpha$.  If $h\in H_0$, then $h=\prod_{i=0}^ny_{l_i}^{q_i}$ with all $q_i\in\mathbb{Q}$ nonzero  and ${l_{i}}<{l_{i+1}}$ for $0\le i<n$.  Then $\log h=\sum_{i=0}^n\ {q_i}\log y_{l_i}$.  Since $\log y_{l_i}=y_{l_i+1}$, $\log h\in B_0$ and $\delta_0(\log h) \in \ek{H_\gamma^{>1}}$.

Suppose the statement holds for all $\lambda<\beta$.  If $\beta$ is a limit ordinal, $h\in H_\beta$ implies $h\in H_\lambda$ for some successor ordinal $\lambda<\beta$, so the statement holds by induction.
Suppose $\beta=\lambda+1$ and $h\in H_\beta-H_\lambda$.
By construction, $h=h'{\prod_{i=0}^{n}} {2^{t_i}}$ where $h'\in H_\lambda$ and $t_i\in B_\lambda$ and $\delta_\lambda(t_i)\in \ek{H_\lambda^{>1}}$  for all $1\le i\le n$.  Then, $\log h= \log h'+\sum_{i=0}^n t_i$   has the desired features by induction.
\end{proof}

To show that the elements $c_{\beta,i}$ get the developments we want for them, we must show that other elements cannot compete for these developments.

\begin{lem}\label{L:Ind} Suppose Ressayre's construction is run on a well ordering $\prec$ such that  $y$ is the first element and the elements $c_{\beta}$ for  $\beta\le \alpha$  form the initial $\omega$ segment.   For all $ \beta, \gamma\le \alpha$, the following statements hold.
\begin{enumerate}
\item\label{lv}  If $\gamma$  is a limit ordinal greater than $\beta$, then $c_{\gamma,i}$ has no development over $H_\beta$.

\item\label{sv}   If $\gamma$ is a successor ordinal greater than $\beta$, then $c_{\gamma,i}$ has no valuation in $H_\beta$.

\item\label{ld}  If $\beta$ is $0$ or a limit ordinal, then $c_{\beta,i}\in B_\beta$ and $c_{\beta,i}$ has the development $\sum_{ i\le j<\omega} c_{\beta_j,j}\in \ek{H_\beta^{>1}}$  where the sequence $(\beta_i)_{i\in\omega}$ is defined as follows.  If $\beta=0$, then $\beta_i=i$ for all $i\in\omega$.  If $\beta$ is a limit ordinal, $(\beta_i)_{i\in\omega}$  is the sequence of successor ordinals converging to $\beta$  given by the  notation for
$\beta$ in the fixed path through $\mathcal{O}$.

\item\label{sd} If $\beta$ is a successor ordinal, then $c_{\beta,i}$ is in $H_\beta^{>1}$.
\end{enumerate}
\end{lem}

\begin{proof}  We begin with the case where $\beta=0$.    Clearly, $c_0$ will be assigned the development $\sum_{1\le i<\omega} y_i$ if it is the first element after $y$ in $\prec$.  However, $c_0$ may not be the first such element; there may be finitely many other $c_\beta$ before $c_0$.
Statements \ref{lv} and \ref{sv} imply that these finitely many $c_\beta$ would not interfere with assigning the development $\sum_{ 1\le i<\omega} y_i$ to $c_0$. Hence, Statements \ref{lv} and \ref{sv} give Statement \ref{ld}.

We begin by showing for all $\gamma>0$ and all $i\in\omega$ that $c_{\gamma,i}$ has no valuation in $H_0$.    Suppose otherwise, and let $\gamma$ be the first ordinal witnessing the failure.  If $\gamma$ is a limit ordinal, then the valuation of $c_{\gamma,i}$ is the same as that of $c_{\gamma_i,i}$, where
$\gamma_i$ is a smaller successor ordinal.  So, we may suppose that $\gamma = \lambda+1$ for some $\lambda$.  The element  $c_{\lambda+1,i}$ was defined to be $2^{c_{\lambda,i+1}}$.   Since $c_{\lambda+1,i}$ has a valuation in $H_0$,
we have that $c_{\lambda+1,i}$ equals
$cy_0^{q_0}y_1^{q_1}\cdots y_n^{q_n}$, where $c$ is finite and the  $q_i\in\mathbb{Q}$.  Then, taking logs, we have $c_{\lambda,i+1} = \log(c) + q_0y_1 + q_1y_2 + \cdots + q_ny_{n+1}$.  We see that
$c_{\lambda,i+1}$ has valuation equal to some $y_i$, which is in $H_0$, since both $c_{\lambda+1,i}$ and $c_{\lambda,i+1}$ are infinite. We must have $\lambda = 0$ and $\gamma = 1$, since otherwise we have reached a contradiction.
If we use a different ordering, putting $c_0$ first after $y$, then $c_0$ would be assigned the desired development.  Then,
$c_{0,i+1}$ would be in
$B_0$ and $c_{1,i} = 2^{c_{0,i+1}}$ would be in $H_1 - H_0$ by Ressayre's construction.  Therefore, $c_{1,i}$ has no valuation in $H_0$.
So, Statements \ref{lv}, \ref{sv},\ref{ld}, and \ref{sd}  hold when $\beta=0$.

Suppose $\beta\le\alpha$ is arbitrary and that the statements in
Lemma \ref{L:Ind} hold  for all $\lambda<\beta$ and all $\gamma\le\alpha$.

We begin by proving  Statement \ref{sd}.  If $\beta=\lambda+1$ is a successor ordinal, we have that $c_{\beta, i}=2^{c_{\lambda, i+1}}$.  If $\lambda$ is itself a successor ordinal, we have that ${c_{\lambda, i+1}}\in H_\lambda^{>1}$ by  Statement \ref{sd}  for $\lambda$ of the induction hypothesis.  Since ${c_{\lambda, i+1}}\in H_\lambda^{>1}$,  $c_{\beta, i}=2^{c_{\lambda, i+1}}\in H_\beta^{>1}$ by construction.  If $\lambda$ is a limit ordinal, we have that
$c_{\lambda,i+1}\in B_\lambda$ is assigned a development in $\ek{H_\lambda^{>1}}$
by  Statement \ref{ld}  for $\lambda$ of the induction hypothesis.  Again, by construction, $c_{\beta, i}=2^{c_{\lambda, i+1}}\in H_\beta^{>1}$ as desired.

We now show that Statements \ref{lv} and \ref{sv} hold for $\beta$ and all $\gamma\le \alpha$ by induction on $\gamma$.
Given some ordinal $\gamma'$, additionally suppose that Statements \ref{lv} and \ref{sv} hold for all $\gamma<\gamma'$ with respect to $\beta$.   First, suppose $\gamma'=\gamma+1$ is a successor ordinal, and suppose for a contradiction that $c_{\gamma', i}=2^{c_{\gamma, i+1}}$ has a valuation in $H_\beta$.
By construction of $H_\beta$,  we have that
\begin{equation}
c_{\gamma', i}=2^{c_{\gamma, i+1}}=ch2^{b_1}\cdots2^{b_n}
\end{equation}
where $c$ is finite,  $h\in H_{\lambda}$, and $b_j\in B_\lambda$ so that $\phi_\lambda(b_j)\in \ek{H_\lambda^{>1}}$ for some ordinal $\lambda<\beta$.   Taking logs of both sides, we have that
\begin{equation}
c_{\gamma, i+1}=\log(c) + \log h+b_1+\ldots+b_n.
\end{equation}
 By Lemma \ref{L:logh}, $\log h\in B_\lambda$ and $\phi_\lambda(\log h)\in \ek{H_{\lambda}^{>1}}$.  Thus, $v(c_{\gamma, i+1})\in H_\lambda$.
If $\gamma$ is a successor ordinal,  this would contradict Statement \ref{sv} of the inductive hypothesis with respect to $\lambda$.  So, suppose $\gamma$ is a limit ordinal.  Consider the sequence $({\gamma_j})_{j\in\omega}$ of successor ordinals given by the notation for $\gamma$ such that $\lim_{j\rightarrow\infty}\gamma_j=\gamma$.  Let $l$ be the least natural number such that $\gamma_l>\lambda$.  We have $c_{\gamma_j, j}\in H_\lambda^{>1}$ for $j<l$ by Statement  \ref{sd} of the inductive hypothesis. Thus,
\begin{equation}
c_{\gamma, i+1}-\sum_{i+1\le j<l}c_{\gamma_j, j}=\log(c) + \log h+b_1+\ldots+b_n-\sum_{i+1\le j<l}c_{\gamma_j, j}
\end{equation}
The left hand side of the equation has the same  valuation as  $c_{\gamma_l, l}$ by definition of $c_{\gamma, i+1}$.  The right hand side of the equation consists of elements whose developments are in $\ek{H_\lambda^{>1}}$ and the finite element $\log c$.  Thus, $v(c_{\gamma_l, l})\in H_\lambda^{>1}$, contradicting Statement \ref{sv} of the inductive hypothesis applied to $\lambda$.  This completes the case where $\gamma'$ is a succesor ordinal.

Second, suppose that $\gamma'$ is a limit ordinal, and suppose for a contradiction that $c_{\gamma', i}$ has a development over $H_\beta$.   Consider the sequence $({\gamma'_j})_{j\in\omega}$ of successor ordinals given by the notation for $\gamma'$ such that $\lim_{j\rightarrow\infty}\gamma'_j=\gamma'$.  Let $l$ be the least natural number such that $\gamma'_l>\beta$.   By  Statement \ref{sd} of the induction hypothesis, we have that $c_{\gamma'_j, j}\in H_\beta^{>1}$ for all $\gamma'_j<\beta$.  If $\beta=\gamma'_j$, we also have $c_{\gamma'_j, j}\in H_\beta^{>1}$ by the proof above of Statement \ref{sd} for $\beta$.  Since $c_{\gamma', i}$ has a development over $H_\beta$,  the difference $c_{\gamma', i}-\sum_{i\le j<l}c_{\gamma'_j, j}$ does as well.  Thus,  $c_{\gamma', i}-\sum_{i\le j<l}c_{\gamma'_j, j}$  has a valuation in $H_\beta$.  Since $c_{\gamma', i}-\sum_{i\le j<l}c_{\gamma'_j, j}$ has the same valuation as $c_{\gamma'_l, l}$ by definition, $v(c_{\gamma'_l, l})\in H_\beta$.  Since $\gamma'_l>\beta$ is a successor ordinal less than $\gamma'$, this contradicts Statement \ref{sv} of the induction hypothesis with respect to $\beta$. This completes our induction on $\gamma'$.  We have proved Statements \ref{lv} and \ref{sv} for $\beta$ and all $\gamma\le\alpha$.

We finally prove Statement \ref{ld} for $\beta$.  Suppose $\beta$ is a limit ordinal.
By Statements \ref{ld} and \ref{sd} of the induction hypothesis, we have that all $c_{\lambda, i}$ for $\lambda<\beta$ receive their desired developments in $B_\lambda$.   In particular, $c_{\beta_j, j}\in H_\beta^{>1}$ for all $j\in\omega$.   We have that no element $c_{\gamma,k}$ for $\gamma>\beta$ has a development over $H_\beta$ by Statements \ref{lv} and \ref{sv} for $\beta$. Since the elements $c_{\gamma}$  are the only elements that could come before $c_\beta$ in the initial $\omega$ segment of the well ordering $\prec$, the element $c_{\beta, i}$ will enter $B_\beta$ and $\phi_\beta(c_{\beta, i})=\sum_{i\le j<\omega} c_{\beta_j, j}\in \ek{H_\beta^{>1}}$.
Thus, Statement \ref{ld} holds for $\beta$.  This completes the proof of Lemma \ref{L:Ind}.

\end{proof}

We now show that $\Gamma_\alpha$ is consistent.
The formulas $\psi$ and $\varphi_\prec$ are satisfied by $R, k$, and $\prec_\alpha$ by construction.
We now show $\varphi_\lambda$ holds for each limit ordinal $\lambda< \alpha$.  If

If $(B_j, H_j, \gamma_j)$ is not in  $L_\lambda$ for some  $j<\lambda$, then $\varphi_\lambda$ holds trivially.
Otherwise,  by Lemma \ref{L:Ind} Statements \ref{lv} and \ref{sv},  there is an element of
\mbox{$B_j - \cup_{\gamma<j} B_\gamma$,} namely $c_j=c_{j, 1}$, so $\varphi_\lambda$ is satisfied.
Thus, $\Gamma_\alpha$ is consistent.

We are in a position to apply Barwise Compactness.  By Theorem \ref{T:BarwiseCompact}, we obtain an $\omega$-model of $KP$ with $R$ and $k$ as elements, and a $\Delta^0_3$ ordering $\prec$ of type $\omega+\omega$ such that if  Ressayre's construction is run on $R$, $k$, and $\prec$, producing a chain of development triples  $(B_i, H_i, \phi_i)_{i<\zeta}$ leading to the first non-trivial maximal and dyadic triple $(R_1, G_1, \delta_1)$, then either some triple $(B_j, H_j, \gamma_j)$ for $j<\zeta$ is not in  $L_{\omega_1^{CK}}$,
or else the length of the chain $\zeta$ is noncomputable.  This completes the proof of Theorem \ref{offcharts}.  Hence, Ressayre's construction on $R$, $k$, and  $\prec$ cannot be completed in $L_{\omega_1^{CK}}$.

Although Ressayre's construction may not be carried out in $L_{\omega_1^{CK}}$,
we can use $\Sigma$-saturation obtain an exponential integer part in a fattening of $L_{\omega_1^{CK}}$.

\begin{prop}\label{fatteningEIP}

Let $R$ be a hyperarithmetical real closed exponential field.  There is an exponential integer part
$Z$ such that \ $(R,Z)$ lives in a fattening of $L_{\omega_1^{CK}}$.  

\end{prop}

\begin{proof}

Since $R$ is hyperarithmetical, it is trivially $\Sigma_A$-saturated.  Let $\Gamma$ be the natural set of sentences saying that $Z$ is an exponential integer part.  By Theorem B, $R$ has an exponential integer part, so the consequences of $\Gamma$ are true in $R$.  Therefore, by Theorem \ref{T:SigmaSat}, there is an exponential integer part $Z$ such that $(R,Z)$ is $\Sigma_A$-saturated.  This means that $(R,Z)$ lives in a fattening of $L_{\omega_1^{CK}}$, with no non-computable ordinals by Proposition \ref{P:SigmaSat}.

\end{proof}

\end{document}